\documentclass[a4paper, 11pt]{article}
\usepackage{latexsym}
\usepackage[fleqn]{amsmath}
\usepackage{amssymb}
\usepackage{amsthm}
\usepackage{enumitem}
\usepackage[normalem]{ulem}
\usepackage{tikz}
\usepackage{tkz-euclide}
\usetikzlibrary{calc,intersections}
\usepackage[dvips,a4paper, hmargin={3cm,3cm},vmargin={2.8cm,2.8cm}]{geometry}
\usepackage{enumitem}
\usepackage{url}
\setlist{itemsep=0.4pt,topsep=0pt}
\usepackage{comment}

\newcommand{\rk}{\mathrm{rk}}
\DeclareMathOperator{\len}{length}
\DeclareMathOperator{\bl}{b}

\newtheoremstyle{note}
{7pt}
{10pt}
{\normalfont}
{}
{\bfseries}
{ }
{.5em}
{}
\theoremstyle{note}
\newtheorem{thm}{Theorem}[section]
\newtheorem{lemma}[thm]{Lemma} 
\newtheorem{df}[thm]{Definition}
\newtheorem{rmk}[thm]{Remark} 
\newtheorem{prop}[thm]{Proposition}

\newtheorem{cor}[thm]{Corollary}

\newtheorem{ex}[thm]{Example}

\title{Endomorphisms of free Steiner quasigroups}
\author{Silvia Barbina,  Enrique Casanovas \thanks{Research partially supported by grants PID2020-116773GB-I00 and PID2023-147428NB-I00 from the Spanish Government. The first author has also been supported by PRIN2022 \textit{Models, sets and classifications}, prot. 2022TECZJA.}}
\begin{document}
		\maketitle
	\begin{abstract}
 A free Steiner quasigroup is a free object in the variety of Steiner quasigroups. Free Steiner quasigroups are characterised by the existence of a levelled construction that starts with a \textit{free base} --- that is, a set of elements none of which is a product of the others, and which generate the quasigroup. Then each element in a free Steiner quasigroup $M$ can be obtained as a term on the free base. We characterise homomorphisms between substructures of a free Steiner quasigroup where one generator is replaced by a term in the original generators. The characterisation depends on certain synctactic properties of the term in question.
	\end{abstract}
	\section{Introduction}

A \textbf{Steiner triple system} (STS) of order $\lambda$ is a pair $(M, \mathcal{B})$ where $M$ is a set with $\lambda$ elements, and 
 $\mathcal{B}$ is a collection of 3-element subsets of $M$ (the \textbf{blocks})
such that any two $x, y \in M$ are contained in exactly one block. If we require that any two $x,y \in M$ are contained in \textit{at most }one block, $M$ is said to be a  \textbf{partial Steiner triple system}.

From a model theoretic standpoint, Steiner triple systems can be viewed as relational structures with a single a ternary relation $R$ interpreted as the set of blocks (that is, three elements $x, y$ and $z$ are in a block if and only if $R(x,y,z)$). Then $R$ must satisfy
\begin{itemize}
\item for all permutations $\sigma \in S_3$, $ \forall x_1 x_2 x_3 \, \left(R(x_1, x_2, x_3) \rightarrow R(x_{\sigma(1)}, x_{\sigma(2)}, x_{\sigma(3)} \right)$
\item $\forall x y z \left( R(x,y,z) \rightarrow y\neq z \right)$
\item $\forall x y w z \left( R(x,y,w) \wedge R(x,y,z) \rightarrow w=z \right)$.
\end{itemize}
With this choice of language, a substructure of a Steiner triple system is a partial Steiner triple system.

An STS $(M, \mathcal{B})$ determines a \textit{Steiner quasigroup}, that is, a pair $(M, \cdot)$ where $\cdot$ is a binary operation such that
\begin{itemize}
\item $\forall x \, (x\cdot x =x)$
\item if $x \neq y$, then $x \cdot y = z$ where $\{x,y,z \} \in \mathcal{B}$.
\end{itemize}
The operation $\cdot$ defined in this way satisfies
 \begin{enumerate}
\item $\forall x \, (x \cdot y= y\cdot x)$
\item $\forall x \, (x\cdot x =x)$
\item $\forall x \, (x\cdot (x\cdot y)=y)$.
\end{enumerate}
Conversely, any structure $(M, \cdot)$ where $\cdot$ satisfies properties 1, 2 and 3 determines an STS $(M, \mathcal{B})$.

In universal algebraic parlance, Steiner quasigroups are algebras. Thus, if $M$ is a Steiner quasigroup, there are standard notions of a \textit{subalgebra} of $M$, and of the subalgebra \textit{generated} by a subset $A \subseteq M$. A subalgebra of a Steiner quasigroup is a substructure in the model theoretic sense, and in particular it is a Steiner quasigroup. If $A \subseteq M$, then $\langle A \rangle_M$ denotes the substructure of $M$ generated by $A$. The subscript $M$ is usually omitted when clear  from the context.

A Steiner quasigroup induces a partial STS structure on each of its subsets, so we often refer to a partial Steiner triple system when we need to consider a subset of a Steiner quasigroup that is not a substructure.

Given a class  $\mathcal{V}$ of algebras in a given signature, we say that a member $U$ of $\mathcal{V}$ generated by a subset $X \subseteq U$ has the  \textbf{universal mapping property} (UMP) over $X$ with respect to $\mathcal{V}$ if for every $V \in \mathcal{V}$, every map
 \[ \varphi : X \to V \]
extends to a homomorphism $ \hat{\varphi} : U \to V $. A standard argument proves that the homomorphism $\hat{\varphi}$ is unique. Such an algebra $U$ is said to be \textbf{freely generated by $X$}, and $X$ is a \textbf{set of free generators} of $U$. 

It is well known that when $\mathcal{V}$ is a variety, given any cardinal $\lambda \geq 1$ there are $U \in \mathcal{V}$ and $X \subseteq U$ such that $U$ is freely generated by $X $ and $|X|= \lambda$.

Since the class $\mathcal{K}$ of Steiner quasigroups is a variety in the universal algebraic sense, for each cardinal $\lambda \geq 1$, there is a Steiner quasigroup that is freely generated by a set of cardinality $\lambda$. Then a \textbf{free Steiner quasigroup} is a Steiner quasigroup that is freely generated by some subset $X$, and so it has UMP over $X$ with respect to the class~$\mathcal{K}$. We assume $\lambda \geq 3$ to exclude the trivial cases of Steiner quasigroups with one and three elements.

In Section~2 we give the definitions of free base and independent set, and we recall various results concerning free bases, including the uniqueness, up to isomorphism, of the Steiner quasigroup freely generated by a base of cardinality $\lambda$. In Section~3, we explain how a free Steiner quasigroup can be seen as a levelled construction from a free base, and we define the related notions of free and hyperfree orderings. We also characterise the existence of a hyperfree ordering via a property that can be expressed by a set of first-order formulas in the language of Steiner triple systems. In Section~4, we introduce the notions of equivalent and reduced terms, and a rank for terms, and in Section~5 we apply these notions to characterise the independence of a tuple in terms of its behaviour with respect to reduced terms. Finally, in Section~6 we describe endomorphisms of  Steiner quasigroups generated by an independent tuple.

\section{Independence and free bases}

In this section we review the definitions of an \textit{independent} set of generators of a Steiner quasigroup and the related notion of free base. For convenience, we then state certain standard results concerning free algebras and show why they hold in the specific case of Steiner quasigroups. Moreover, we clarify the relationship between free bases, sets of generators of minimal cardinality, and minimal sets of generators. We use these results to show that free Steiner quasigroups are determined up to isomorphism by the cardinality of a free base.
	
	\begin{df} Let $(M,\cdot)$ be a Steiner quasigroup. A subset $A\subseteq M$ is  \textbf{independent} if the generated substructure $\langle A\rangle_M$ is freely generated from $A$, that is, if $\langle A\rangle_M$ has UMP over $A$ for the class of Steiner quasigroups. 
		
		We say that $A\subseteq M$ is a \textbf{free base} from $M$  if it is independent and it generates $M$. 
	\end{df}
	When we say $\bar{a} = \langle a_i : i \in I \rangle$ is independent, we understand $a_i\neq a_j$ if  $i,j\in I$ and $i\neq j$.
	
	If $M$ is a Steiner quasigroup generated by the subset $A$ and $\bar{a}$ is an enumeration of $A$, then for every $b \in M$ there is a term $t(\bar{x})$ such that $b=t(\bar{a})$.
	
	\begin{rmk}\label{May11_1}  If $\bar{a},b$ is independent, then  $b\not\in\langle \bar{a}\rangle$.
	\end{rmk}
	\begin{proof} Assume  there is a term $t(\bar{x})$ such that $t(\bar{a})= b$. Let $N$ be a  Steiner quasigroup and let $b_1,b_2 \in N$ be distinct elements. Then $b_2\neq b_1=t(b_1,\ldots,b_1)$. The mapping sending all elements of $\bar{a}$ to $b_1$  and  $b$ to $b_2$ cannot be extended to an homomorphism.
	\end{proof}

The next proposition is a general result that relates the cardinalities of free bases and generating sets in free algebras.  We state the result for Steiner quasigroups.

\begin{prop}[Fujiwara, Jónsson-Tarski] \label{May11_2} If $A$ is a free base of the Steiner quasigroup $M$ and $B\subseteq M$ is a set of generators of $M$, then  $|B|\geq |A|$.
	\end{prop}
	\begin{proof} 
	 If $|A|=\kappa\geq \omega$, the result follows from Theorem~I in \cite{fujiwara}. If $|A|$ is finite, the result follows from the fact that there is a finite Steiner quasigroup of cardinality $>1$ and Theorem~1 in \cite{jontarski}. \end{proof}	
	
We now show that the UMP for the class of all Steiner quasigroups is equivalent to the UMP for the class of all finite Steiner quasigroups. This result hinges on the fact that a finite partial Steiner triple system can always be embedded in a finite STS. There are several ways to construct such a finite extension - see, for example,  \cite{andersenetal}, \cite{lindner} and~\cite{treash}. As a consequence, any finite subset of a Steiner quasigroup can be embedded in a finite Steiner quasigroup.

	\begin{lemma} \label{May11_4} If an equation  $t(\bar{x})=r(\bar{x})$ holds in all finite Steiner quasigroups, then it holds in all Steiner quasigroups.
	\end{lemma}
	\begin{proof} Assume $M$ is a  Steiner quasigroup, $\bar{a}\in M$ is a finite tuple and  $t,r$ are terms such that  $t(\bar{a})\neq r(\bar{a})$.  Let   $M_0\subseteq M$ be a partial finite Steiner triple system  such that $s(\bar{a})\in M_0$ for every subterm $s$ of $t$ or $r$. In particular, $t(\bar{a}),r(\bar{a})\in M_0$. Then $M_0$ can be extended to a finite quasigroup $M_1$.  Note that for every subterm $s$ of $t$ or $r$, $s(\bar{a})$ computed in $M$ or in $M_1$ is the same. Then the equation $t(\bar{a}) = r(\bar{a})$  does not hold in $M_1$.
	\end{proof}
	
	\begin{lemma}\label{May11_5} Let $M,N$ be Steiner quasigroups, let $A\subseteq M$ be a set of generators and let $f:A\rightarrow N$ be a mapping. It is possible to extend $f$ to some homomorphism $f^\prime: M\rightarrow N$  if and only if for every equation $t(\bar{x})= r(\bar{x})$, for every tuple $\bar{a}$ of different elements of $A$ such that $M\models t(\bar{a})= r(\bar{a})$ we have that $N\models t(f(\bar{a}))= r(f(\bar{a}))$.
	\end{lemma}
	\begin{proof} The extension is defined by  $t(\bar{a})\mapsto t(f(\bar{a}))$ and the assumption can be used to check that this is well defined.
	\end{proof}
	
	\begin{lemma}\label{May11_6} If a Steiner quasigroup $M$ is free over $A\subseteq M$ with respect to the class of all finite Steiner quasigroups, then it is free over $A$ with respect to all Steiner quasigroups.
	\end{lemma}
	\begin{proof} We assume $A$ is a set of generators of $M$ and for every finite Steiner quasigroup $N$, every mapping $f:A\rightarrow N$  can be extended to some homomorphism $f^\prime: M\rightarrow N$. We check that $A$ is a free base of $M$. Let $N$ be a Steiner quasigroup and let $f:A\rightarrow N$ be a mapping.  We use  Lemma~\ref{May11_5} to show that we can extend $f$ to a homomorphism. If the extension does not exist, there is an equation $t(\bar{x})= r(\bar{x})$ and a tuple $\bar{a}$ of distinct elements of $A$ such that $M\models t(\bar{a})= r(\bar{a})$ but $N\models t(f(\bar{a}))\neq r(f(\bar{a}))$. By Lemma~\ref{May11_4}, there is a finite Steiner quasigroup $N_0$ and a tuple $\bar{c}\in N_0$ such that  $N_0\models t(\bar{c})\neq r(\bar{c})$.  We can define a mapping $g:A\rightarrow N_0$  such that $g(\bar{a})= \bar{c}$. By assumption, $g$ can be extended to some homomorphism $g^\prime:M\rightarrow N_0$, which contradicts Lemma~\ref{May11_4}.
	\end{proof}
	
	\begin{prop}[Jónsson-Tarski]\label{May11_7} Assume $A$ is a finite free base of the Steiner quasigroup $M$. If $B\subseteq M$ is a set of generators  and $|A|=|B|$, then $B$ is a free base of $M$.
	\end{prop}
	\begin{proof} By Lemma~\ref{May11_6} and Theorem~2 in \cite{jontarski}. 
	\end{proof}
	
	\begin{cor}\label{May11_8}  Assume $A$ is a finite free base of the Steiner quasigroup $M$ and  $B\subseteq M$ is a set of generators. Then $|A|=|B|$ if and only if $B$ is a free base of $M$.
	\end{cor}
	\begin{proof} This follows from Proposition~\ref{May11_7} and Proposition~\ref{May11_2}.
	\end{proof}
	
	\begin{cor}\label{May11_9} Let $M$ be a finitely generated free Steiner quasigroup. If $A \subseteq M$ is a set of generators of minimal cardinality, then $A$ is a free base of $M$. 
	\end{cor}
	\begin{proof} Since $A$ has minimal cardinality, $|A|$ is finite, so Proposition~\ref{May11_2} gives $|A| \geq |B|$, where $B$ is a base of $M$. It follows that $|A|=|B|$ by minimality of $|A|$. By Proposition~\ref{May11_7}, $A$ is a free base.
	\end{proof}
	
	\begin{cor}\label{May11_9.1} Let $M,N$ be Steiner quasigroups freely generated over $A\subseteq M$ and $B\subseteq N$ respectively. Then $M\cong N$ if and only if $|A|=|B|$. In fact, every bijection between $A$ and $B$ can be extended  to an isomorphism between $M$ and $N$.
	\end{cor}
	\begin{proof} The direction from right to left is a standard argument for free structures:  if  $f:A\rightarrow B$ is a bijection then it can be extended to some homomorphism $\hat{f}: M\rightarrow N$  and $f^{-1}$ can be extended to a homomorphism  $f^\prime: N\rightarrow M$.  Then $f^\prime\circ \hat{f}$  must be the identity on $M$  and it follows that $\hat{f}$ is an isomorphism.
		
		For the other direction, we assume $|A|<|B|$ and $M\cong N$ and we seek a contradiction. We may assume $M=N$. Hence $A$ and $B$ are free bases of $M$. This contradicts Proposition~\ref{May11_2}. 
	\end{proof}
	
	\begin{cor}\label{May11_9.2}
		If $M$ is a finitely generated free Steiner quasigroup and $f:M\rightarrow M$ is a surjective homomorphism, then $f$ is an automorphism of $M$.
	\end{cor}
	\begin{proof} Let $A$ be a base of $M$. Since $f(A)$ is a finite set of generators of $M$, Proposition~\ref{May11_2} implies $|f(A)|\leq |A|$  and therefore $|f(A)|=|A|$. By Corollary~\ref{May11_8}, $f(A)$ is a free base of $M$. By Corollary~\ref{May11_9.1}, $f\restriction A$ extends to some automorphims of $M$. Since $f\restriction A$ has a unique extension to a homomorphism from $M$ into $M$, the extension is $f$.
	\end{proof}

	\section{Hyperfree and free extensions}
Free Steiner quasigroups can be constructed in stages, and the construction gives rise to two natural notions of ordering which generalise to the non-free case.  These notions were first formulated in \cite{sieben} for projective planes and they have a natural analogue in the context of Steiner triple systems. In particular, the existence of an \textit{HF-ordering} for a partial Steiner triple system $A$ is equivalent to a condition on $A$ that can be expressed by a set of first-order formulas. 

Many of the definitions and results in this section  have essential applications for the axiomatisation of the first-order theory of free Steiner triple systems \cite{barcas3}. 
	
Any Steiner quasigroup can be thought of as the union of a chain of partial Steiner triple systems where the first element of the chain is a set of generators, and each subsequent element contains previously undefined products. The existence of a free generating set is equivalent to certain specific properties of this chain, described in Proposition~\ref{May12_2} below.

 Let $M$ be a Steiner quasigroup generated by $A\subseteq M$.  Then  $M=\bigcup_{n<\omega} S_n$ for a chain of sets $S_n$ such that  $S_0 =A$ and $S_{n+1}= \{a\cdot b\mid a,b\in S_n\}$. Note that $S_n\subseteq S_{n+1}$  since  $a\cdot a = a$.
 We will write $S_{<n}$ for $\bigcup_{i<n} S_i$. 
	 The sets $S_n$  are called the \textbf{levels} of $M$ over $A$. An element $a\in M$ has   \textbf{level~$n$ over $A$}  if $a\in S_n\smallsetminus S_{<n}$.

	\begin{prop}\label{May12_2} Let $M$ be a Steiner quasigroup and let $A\subseteq M$. The following are equivalent:
		\begin{enumerate}
			\item $A$ is a free base of $M$.
			\item $M$ is generated by $A$ with  levels $(S_n\mid n<\omega)$ such that $A=S_0$ and
			\begin{enumerate}
				\item if $a,b\in S_n\smallsetminus S_{<n}$ and $a\neq b$, then $a\cdot b\not\in S_n$.
				\item if $a\in S_{n+1}\smallsetminus S_n$, there is a unique pair $\{b,c\}\subseteq S_n$ such that $b\neq c$ and $a= b\cdot c$.
			\end{enumerate}
		\end{enumerate}
		Moreover, property (b) implies that if  $a\in S_{n+1}\smallsetminus S_n$, and $b, a\cdot b\in S_n$, then  $\{b,a\cdot b\}$ is the unique pair given by (b) for  $a$.
	\end{prop}
	\begin{proof} 2 $\Rightarrow$ 1.  Let $N$ be a Steiner quasigroup and let $f:A\rightarrow N$ be a mapping. We extend $f$ to a homomorphism $\hat{f}: M\rightarrow N$.  We obtain $\hat{f}$ as the union $\hat{f}= \bigcup_{n<\omega}f_n$  of a chain  of partial homomorphisms $f_n : S_n \to N$. This means that  $f_n(a\cdot b)= f_n(a)\cdot f_n(b)$  whenever $a,b$ and $a\cdot b$  belong to $\mathrm{dom} (f_n)=S_n$. We start with $f_0=f$.  Recall that $A=S_0$ and that $a\cdot b\not\in S_0$  if $a,b\in S_0$ and, therefore, $f_0$ is a partial homomorphism with domain $S_0$. Now we extend  $f_n:S_n\rightarrow N$  to $f_{n+1}$. 
	For $a \in S_{n+1}\smallsetminus S_n$, there is a unique pair $\{b,c\}\subseteq S_n$ of distinct elements $b,c$ with $a=b\cdot c$. We let  $f_{n+1}(a) = f_n(b)\cdot f_n(c)$.  By (a) and (b), the only defined products involving $a$ in $S_{n+1}$ are   $a=b\cdot c=c\cdot b$, $a\cdot a=a$, $a\cdot b = b\cdot a= c$ and $a\cdot c= c\cdot a =b$ and $f_{n+1}$ respects these products.
		
		1 $\Rightarrow$ 2.  It is easy to construct a Steiner quasigroup $M^\prime\supseteq A $ with the properties described in 2. It is enough to start with $S^\prime_0=A$ and inductively define $S^\prime_{n+1}$ by adding to $S^\prime_n$ a different element $a\cdot b$ for each pair $\{a,b\}$ of elements of $S^\prime_n$ whose product is not defined in $S^\prime_n$. The only  products involving  the new element $c=a\cdot b$ in $S^\prime_{n+1}$ are $c=a\cdot b=b\cdot a$, $c\cdot c=c$, $c\cdot a= a\cdot c=b$ and $c\cdot b=b\cdot c= a$. We have proven that $A$ is a free base of $M^\prime= \bigcup_{n<\omega}S^\prime_n$. The uniqueness of free  Steiner quasigroups over a given set of generators guarantees that $M$ and $M^\prime$ are isomorphic over $A$.  The isomorphism can be used to obtain the levels $S_n$ in $M$ with the required properties.
	\end{proof}
	
		\begin{df} Let $M$ be a Steiner quasigroup, and let  $A\subseteq M$ be a free base. The sequence $(S_n\mid n<\omega)$ of levels described in Proposition~\ref{May12_2} will be called a \textbf{standard free construction of $M$ over $A$}.
	\end{df}

A standard free construction can be thought of as a construction in stages and suggests the following notion of ordering.
 
	\begin{df}
		A \textbf{free ordering}  (or an \textbf{F-ordering})   of a Steiner quasigroup $M$ over $A\subseteq M$ is a linear ordering $<$ of $M$ such that for every $a\in M$  there is a unique pair $\{a_1,a_2\}\subseteq M$ with $a=a_1\cdot a_2$,  and for each $i$, either $a_i\in A$ or $a_i<a$. 
		
		Note that, when $a\in A$, the unique pair is given by $a_1=a_2=a$.
	\end{df}
	
	\begin{rmk}\label{May12_12.0} Let $M$ be a Steiner quasigroup with standard free construction $(S_n\mid n<\omega)$ over $A$. Let $<_n$ be a linear ordering of $S_n\smallsetminus S_{<n}$, and for $a,b\in M$ define
	$$a<b \Leftrightarrow \mbox{ for some } n, \ b\in S_n\smallsetminus S_{<n} \mbox{ and either } a\in S_{<n} \mbox{ or } a\in S_n\smallsetminus S_{<n} \mbox{ and } a<_n b.  $$ 
		Then $<$ is an F-ordering of $M$ over $A=S_0$, and $S_0$ is an initial segment. If every $<_n$ is a well-order, then $<$ is a well-order.
		\end{rmk}
		
	\begin{rmk}\label{May12_12.1}
		 Let $M$ be a Steiner quasigroup. If there is an F-ordering of $M$ over $A$, there is another one where $A$ is an initial segment. If the original order is a well-order, the new one is also a well-order and, moreover, $M=\langle A\rangle$.
	\end{rmk}	
	
	\begin{df} A \textbf{hyperfree ordering}  (or \textbf{HF-ordering}) of a partial Steiner triple system $A$ is a linear ordering $<$ of $A$ such that for every $a\in A$ there is at most one pair of smaller elements $\{a_1,a_2\}\subseteq A$ such that $a=a_1\cdot a_2$, that is, there is at most one block $\{a,a_1,a_2\}$ with $a_1,a_2<a$.
	\end{df}
The next proposition characterises free bases in terms of the existence of specific F-orderings and HF-orderings. Recall that a well-founded linear order is a well-order.
	\begin{prop}\label{May12_2.2} Assume $M$ is a Steiner quasigroup and $A\subseteq M$. The following are equivalent:
		\begin{enumerate}
			\item $A$ is a free base of $M$.
			\item There is an F-ordering of $M$ over $A$ which is a well-order.
			\item There is an HF-ordering of $M$ which is a well-order and $A$ is the set of points that are not a product of smaller elements.
		\end{enumerate}
		\end{prop}
		\begin{proof}
			1 $\Rightarrow$ 2. By Proposition~\ref{May12_2}, there is standard free construction $(S_n\mid n<\omega)$ of $M$ over $A$. Choose a well-order $<_n$ of every set $S_n\smallsetminus S_{<n}$. By Remark~\ref{May12_12.0}, there is an F-ordering $<$ of $M$ over $A$ which is a well-order extending all orders $<_n$.
			
			2 $\Rightarrow$ 3. By Remark~\ref{May12_12.1} there is an F-ordering $<$ of $M$ over $A$ which is a well-order and $A$ is an initial segment. We know that no element of $A$ is a product of smaller elements. If $a\in M\smallsetminus A$, then $a= a_1\cdot a_2$ where each $a_i$ is either smaller than $a$ or an element of $A$. Since $A$ is an initial segment, $a_1,a_2 <a$.
			
			3 $\Rightarrow$ 1.  Let $<$ be a well-order and an HF-ordering of $M$ and assume $A$ is the set of elements which are not a product of smaller elements. As in Remark~\ref{May12_12.1}, we can modify the order and make $A$ an initial segment. Hence, every element of $M\smallsetminus A$ is a product of smaller elements. By induction on the well-order it is easy to see that $M=\langle A\rangle$. Note that no element of $A$ is a product of two different elements of $A$. Let $N$ be a Steiner quasigroup and $f:A\rightarrow N$ a  mapping. By recursion on the well-order, it is easy to extend $f$ to some homomorphism $f^\prime:M\rightarrow N$.
		\end{proof}

	\begin{prop} \label{freesubs} 
	\begin{enumerate}	
		\item A substructure $M$ of a free Steiner quasigroup $N$ is free. 
		\item If $A$ is a finite subset of the free Steiner quasigroup $M$, then $A$ is contained in a finitely generated free quasigroup.
		\end{enumerate}
	\end{prop}
	\begin{proof}  1. Assume $A$ is a free base of $N$.  By Proposition~\ref{May12_2.2}, there is an HF-ordering of~$N$ which is a well-order and $A$ is the set of elements which are not a product of smaller elements. If we delete from the ordering the elements of $N\smallsetminus M$, we obtain a well-order of $M$ which is an HF-ordering. Let $B\subseteq M$ be the set of elements which are not a product of smaller elements. By Proposition~\ref{May12_2.2}, $B$ is a free base of $M$.
		
			2. By 1, since $\langle A\rangle_M$ is a finitely generated substructure of the free Steiner quasigroup~$M$.
		\end{proof}

	\begin{df} A partial  STS $A$ is \textbf{unconfined}, or \textbf{contains no confined configuration}, if for all finite non-empty $A^\prime \subseteq A$ there is $a \in A^\prime$ that belongs to at most one block in $A^\prime$, that is, $a$ is the product of at most two elements of $A^\prime$.
		
	\end{df}
	
	\begin{prop}\label{2024_nov1} 
	 Free Steiner quasigroups are unconfined Steiner triple systems.
			\end{prop}
	\begin{proof}
	Suppose $A$ is a finite non-empty partial subsystem of $M$. Let $(S_n \mid n < \omega)$ be a standard free construction of $M$. If $A \subseteq S_0$, the claim is clear. Otherwise, since $A$ is finite there is $n \in \omega$ such that 
			\begin{itemize}
				\item $A \subseteq S_{n+1}$
				\item $A \smallsetminus S_n \neq \emptyset$\,.
			\end{itemize}
			Let $a \in A \smallsetminus S_{n}$. Then $a$ is the product of exactly two points in $S_n$. If these two points are in $A$, then $a$ is in one block.  If one of the two points is not in $A$, then $a$ is not in a block. There are no other possibilities.
				\end{proof}
				
We can now connect unconfinedness and HF-orderings.
			
\begin{prop}\label{2024_nov3} The following are equivalent for any partial Steiner triple system $A$:
\begin{enumerate}
\item $A$ is unconfined.
\item There is an HF-ordering of $A$.
\end{enumerate}	
\end{prop}	
\begin{proof} Consider first the case $A$ finite. By induction on $|A|$ it is easy to prove that, if $A$ is unconfined, then it has an HF-ordering. On the other hand, if there is an HF-ordering  of $A$, for any non-empty $A' \subseteq A$ the greatest element $a\in A'$ belongs to  at most one block of $A$, hence to at most a block of $A'$. Consider now the general case. If $A$ has an HF-ordering, every finite subset of $A$ has an HF-ordering and so it is unconfined, which implies that $A$ is unconfined. For the other direction, we use propositional compactness. Assume that every finite subset  of $A$ has an HF-ordering. For every pair  $(a,b)\in A\times A$, let $p_{ab}$ be a corresponding propositional variable. The set of sentences  
$$ \{p_{ab}\wedge p_{bc}\rightarrow p_{ac}\mid a,b,c\in A\}\cup \{\neg p_{aa}\mid a\in A\}\cup \{p_{ab}\vee p_{ba}\mid a,b\in A\}$$
 defines a linear ordering on $A$. If we add 
 $$\{\neg (p_{ba}\wedge p_{ca}\wedge p_{b^\prime a}\wedge p_{c^\prime a} )\mid \{a,b,c\}\neq\{a,b^\prime,c^\prime\} \text{ blocks of } A\}$$ 
 we are defining an HF-ordering on $A$. Every finite subset is satisfiable and compactness gives a valuation that satisfies the whole set.
\end{proof}
			
A notion of \textit{rank} of a partial Steiner triple system is defined in~\cite{sieben}, which corresponds to a natural choice for a predimension function.

\begin{df} Let $A$ be a finite partial STS and let $\bl(A)$ be the set of blocks in $A$. Then the \textbf{rank}, or \textbf{predimension}, of $A$ is
	\[ \delta(A)= |A|-|\bl(A)| \, .\]
\end{df}

It is easy to see that the predimension of a finite Steiner quasigroup is negative except in the case of the Steiner quasigroup of order 7 (the Fano plane), which has predimension 0. On the other hand, the predimension of a finite unconfined partial Steiner triple system is always non-negative (see \cite{barcas3} for details).

 Let $Aa$ be a partial Steiner triple system, with $a\not\in A$. Then 
$$\delta(Aa)= \delta(A)+1-|\bl_a(A)| \, ,$$
 where $\bl_a(A)$ is the set of blocks containing $a$ and intersecting $A$. Hence, if there is at most one such block, $\delta(Aa)\geq \delta(A)$. It follows that if $A\subseteq B$ are finite partial Steiner triple systems and there is an HF-ordering of $B$ with initial segment $A$, then $\delta(B)\geq\delta(A)\geq 0$.

\begin{prop}\label{2024_nov4}
	Finitely generated unconfined Steiner triple systems are free Steiner quasigroups.
\end{prop}
\begin{proof} Let $A$ be a finite set of generators of the unconfined Steiner quasigroup $M$. Let $(S_n\mid n<\omega)$ be the levels of $M$ over $A$. Every $S_n$ is finite. Each element of  $S_{n+1}\smallsetminus S_n$ is in at least one block intersecting $S_n$, and hence  there is an enumeration $(a_i\mid i<k)$ of $S_{n+1}\smallsetminus S_n$ such that  $\delta(S_na_1,\ldots,a_{i+1})\leq \delta(S_na_1,\ldots,a_i)\leq \delta(S_n)$. Concatenating  these enumerations after an arbitrary enumeration of $A=S_0$ gives an enumeration $(a_n\mid n<\omega)$ of $M$. Since $M$ is unconfined, $\delta(a_0,\ldots,a_n)\geq 0$  for every finite~$n$. Therefore there is some $n$  such that, for every $m\geq n$,\,  $\delta(a_0,\ldots,a_m)= \delta(a_0,\ldots,a_n)$. Since $M$ is unconfined, $A_m=\{a_i\mid i\leq m\}$ has an HF-ordering, and we may assume that the elements of $A_m$ which are not a product of smaller elements form an initial segment $A^\prime_m\subseteq A_m$. We start with this HF-ordering of $A_m$ and add the enumeration $(a_i\mid i>m)$, obtaining an F-ordering of $M$ over $A^\prime_m$ which is a well-order. Then $A^\prime_m$ is a free base of $M$ by Proposition~\ref{May12_2.2}.

\end{proof}

	\section{Reduced terms}
	
If $M$ is a Steiner quasigroup generated by a set $A$ and $\bar{a}$ is an enumeration of $A$, then for every $b \in M$ there is a term $t_b(\bar{x})$ such that $b=t_b(\bar{a})$ -- and, in general, the term $t_b$ is not unique. For instance, two terms that can be obtained from one another by commuting one or more subterms are a trivial example of the non-uniqueness. There is also no canonical way to write $b$ as a term in $\bar{a}$, but we can define a notion of a \textit{reduced term} as one that contains no `redundant' products (e.g. of the form $s(ts)$, or $ss$ for subterms $s$ and $t$, etc.). We also introduce a \textit{rank} for terms which measures their syntactic complexity. 
 When $A$ freely generates $M$, for a given $b$ the rank of a reduced term $t$ such that $b=t(\bar{a})$ gives useful information about the level of $b$ in $M$.

	\begin{df} We define an equivalence relation $\sim$ on the set of terms in such a way that $t_1 \sim t_2$ if $t_2$ is obtained from $t_1$ by one or more applications of commutativity to subterms.
		
		A term is said to be \textbf{reduced} if it contains no occurrences of subterms of any of the following forms
		\begin{itemize} 
			\item $t_1 t_2$ with $t_1 \sim t_2$
			\item  $t_1(t_2t_3)$ with $t_1 \sim t_2$ or $t_1 \sim t_3$
			\item $(t_1t_2)t_3$ with $t_1 \sim t_3$ or $t_2 \sim t_3$.
		\end{itemize}
	\end{df}
	
	
	\begin{df} The \textbf{rank} $\rk(t)$ of a term $t$ is defined to be
		\begin{itemize}
			\item $0$ if $t$ is a variable
			\item $\mathrm{max} \{ \rk (r), \rk(s) \} + 1$ if $t = r\cdot s$.
		\end{itemize}
	\end{df}
	
	\begin{lemma}\label{May12_3}\begin{enumerate}
			\item If $t_1\cdot t_2\sim r_1\cdot r_2$, then either  $t_1\sim r_1$ and $t_2\sim r_2$,  or  $t_1\sim r_2$ and $t_2\sim r_1$.
			\item if $t\sim r$ and $t^\prime$ is a subterm of $t$ then $t^\prime\sim r^\prime$ for some subterm $r^\prime$ of $r$.
			\item If $t_1\sim t_2$, then $\rk(t_1)= \rk(t_2)$  and  $t_1(\bar{a})= t_2(\bar{a})$  for any tuple $\bar{a}$.
			\item If  $t_1\sim t_2$ and $t_1$ is reduced, then $t_2$ is reduced.
			\item Any subterm of a reduced term is reduced.
			\item Assume $t,t^\prime$ are reduced, $t\not\sim t^\prime$ and there is no term $r$ such that  $t\sim t^\prime\cdot r$ or $t^\prime\sim t\cdot r$. Then $t\cdot t^\prime$ is reduced.
		\end{enumerate}
	\end{lemma}
	\begin{proof} 1 can be easily checked by induction on the number of applications of commutativity needed for $t_1\cdot t_2\sim r_1\cdot r_2$.  2 is proved by induction on $t$.  3 is clear.  4 follows from 2. 5  and 6 are clear.
	\end{proof}

The next result clarifies how the levels in a construction of a Steiner quasigroup $M$ are linked to the ranks of the reduced terms giving the elements of $M$.
	\begin{lemma}\label{May12_4} Let $M=\langle \bar{a}\rangle$ be a  Steiner quasigroup, with a construction  with levels $(S_n\mid n<\omega)$.
		\begin{enumerate}
			\item If $t$ is a term such that $\rk(t)\leq n$, then $t(\bar{a})\in S_n$.
			\item If $b\in S_n$, then  $b= t(\bar{a})$  for some reduced term $t$ of rank $\rk(t)\leq n$.
			\item  $S_n= \{t(\bar{a})\mid \rk(t)\leq n\} = \{t(\bar{a})\mid t \mbox{ is reduced and } \rk(t)\leq n\}$.
		\end{enumerate}
	\end{lemma}
	\begin{proof} 1 is proved by induction on $n$.  3 follows from 1 and 2. We prove 2 by induction on~$n$.  Assume $n=0$.  If $\bar{a}= \langle a_i : i \in I \rangle$, then $b =a_k$  for some $k \in I$ and $t=x_k$ (a variable, hence reduced and of rank $0$) is the required term. For the inductive step, assume $b\in S_{n+1}$.  If $b\in S_n$ we apply the induction hypothesis. If $b\not\in S_n$, then $b= b_1\cdot b_2$ for $b_1,b_2\in S_n$.  By induction hypothesis, there are reduced terms  $t_1,t_2$ of rank $\leq n$ such that $b_1= t_1(\bar{a})$ and $b_2= t_2(\bar{a})$.  If $t_1\sim t_2$, then $t_1(\bar{a})=t_2(\bar{a})$ and then $b= b_1\cdot b_2 = b_1\cdot b_1 = b_1 \in S_n$, a contradiction. Hence, $t_1\not\sim t_2$.  If there is a term $r$ such that $t_1\sim t_2\cdot r$, then  $b= b_1\cdot b_2=(t_2(\bar{a})\cdot r(\bar{a}))\cdot b_2 = (b_2\cdot r(\bar{a}))\cdot b_2 = r(\bar{a})$. But $\rk(r) <\rk(t_2\cdot r) =\rk(t_1)\leq n$, which implies $b\in S_n$, a contradiction. A similar contradiction follows from the assumption that $t_2\sim t_1\cdot r$ for some term $r$. By  Lemma~\ref{May12_3}, $t= t_1\cdot t_2$ is reduced. Moreover $\rk(t)\leq n+1$  and  $b= t(\bar{a})$.
	\end{proof}

	\section{Terms and independence}

In a standard free construction of a Steiner quasigroup $M$ over the free base $\bar{a}$, the rank of a reduced term $t$ determines the level of the element $t(\bar{a})$. Moreover, the independence of $\bar{a}$ can be characterised in terms of the behaviour of $\bar{a}$ with respect to non-equivalent reduced terms.

	\begin{prop}\label{May19_1}
		Let $M=\langle \bar{a}\rangle$ be a free Steiner quasigroup with standard free construction $(S_i\mid i<\omega)$, let  $\bar{a}= \langle a_i : i \in I \rangle$ be  independent, and let $t= t(\bar{x})$ be a reduced term of rank $k$. Then
		\begin{enumerate}
			\item $t(\bar{a})\in S_k\smallsetminus S_{<k}$
			\item if $r$ is reduced and  $r(\bar{a})= t(\bar{a})$, then $r\sim t$.
		\end{enumerate}
	\end{prop}
	\begin{proof} We prove 1 and 2 jointly by induction on $k$, but with the extra assumption that $\rk(r)=\rk(t)$. This assumption can be eliminated  later since 1 excludes other possibilities.
		
		\underline{Base case $k=0$}.  If $\rk(t)=0$, then $t$ is a variable $x_i$ and  $t(\bar{a})= a_i\in S_0$.  Moreover, if $r$ has rank 0, $r$ is a variable $x_j$  and  $t(\bar{a})= r(\bar{a})$  implies $a_i =a_j$ and therefore  $x_i = x_j$. Hence  $t\sim r$.
		
		\underline{Inductive step $k+1$}.  Let $t= t_1\cdot t_2$, with $k_1= \rk(t_1)\leq \rk(t_2) = k$ and let  $b= t(\bar{a})$.  By inductive hypothesis, $b_1= t_1(\bar{a})\in S_{k_1}\smallsetminus S_{<k_1}$ and $b_2= t_2(\bar{a})\in S_k\smallsetminus S_{<k}$. Then $b= b_1\cdot b_2$ and there are different cases.
		
		\emph{Case 1}. Assume $k_1 = k$.  If  $b_1 = b_2$, then $b= b_1= b_2$ and by inductive hypothesis $t_1\sim t_2$. But then $t$ is not reduced. Hence, 
		$b_1\neq b_2$. By construction, $b\in S_{k+1}\smallsetminus S_k$.  Moreover, if $r$ is reduced of rank $k+1$ with $r(\bar{a})= b$, then $r= r_1\cdot r_2$, with $r_1,r_2$ reduced and, without loss of generality, $\rk(r_1)\leq \rk(r_2)= k$. By inductive hypothesis $b_1^\prime = r_1(\bar{a})\in S_{\rk(r_1)}\smallsetminus S_{<\rk(r_1)}$  and $b_2^\prime = r_2(\bar{a})\in S_k\smallsetminus S_{<k}$. Then $b= b_1^\prime\cdot b_2^\prime$ and by construction $\{b_1,b_2\}= \{b_1^\prime,b_2^\prime\}$. If $b_1= b_1^\prime$ and $b_2 = b_2^\prime$, then $\rk(r_1)= k_1$ and by inductive hypothesis $t_1\sim r_1$ and $t_2\sim r_2$. Therefore in this case $t\sim r$. If $b_1= b_2^\prime$ and $b_2= b_1^\prime$, then $k=k_1= \rk(r_1)$ and, by inductive hypothesis, $t_1\sim r_2$ and $t_2\sim r_1$. Again we conclude that $t\sim r$.
		
		\emph{Case 2}. Assume $k_1 <k$ and $b\not \in S_{<k}$.  Note that $b\neq b_2$ and $b\not \in S_k$. By construction, $b= b_1\cdot b_2 \in S_{k+1}$.  Now let $r$ be reduced, of rank $k+1$ and such that $r(\bar{a})= t(\bar{a})=b$. Exactly as in the previous case we conclude that $r\sim t$.
		
		\emph{Case 3}. The last case is $k_1 < k$ and  $b\in S_j\smallsetminus S_{<j}$  for some $j<k$.  We show this is impossible.   We know that there is some reduced term $s$ of rank $j$ such that $s(\bar{a}) = b$. Now we claim that $s\cdot t_1$ is reduced. Otherwise, there are two possible cases:
		\begin{enumerate}
			\item[(i)] $s\sim t_1$.  But then $b= b_1$ and hence $b=b_1 = b_2$, which is not the case.
			\item[(ii)]  $t_1\sim s\cdot r$ or $s\sim t_1\cdot r$  for some (reduced)  term $r$.  Then $r(\bar{a})= t_1(\bar{a})\cdot s(\bar{a})= b_1\cdot b= b_2$  and by inductive hypothesis $\rk(r)=k$, contradicting $\rk(t_1)=k_1<k$ and $\rk(s)= j<k$.
		\end{enumerate}
		Therefore, we conclude that $s\cdot t_1$ is reduced. Then, since $\rk(s\cdot t_1)\leq k$ and $s\cdot  t_1 (\bar{a}) = b\cdot b_1 = b_2$, we can apply the inductive hypothesis to get $\rk(s\cdot t_1)= k$  and  $s\cdot t_1\sim t_2$. Then $t\sim t_1\cdot t_2\sim t_1\cdot (s\cdot t_1)$, which says that $t$ is not reduced.
	\end{proof}

From now on, tuples and subsets are understood to be in a host Steiner quasigroup $M$.

	\begin{cor}\label{May19_2}  If $\bar{a}$ is independent, then $S_n\smallsetminus S_{<n} =\{t(\bar{a})\mid t \mbox{ is reduced and } \rk(t) = n\}$.
	\end{cor}
	\begin{proof} By Proposition~\ref{May19_1}(1) and Lemma~\ref{May12_4}.
	\end{proof}

Proposition~\ref{May19_1} ensures that non-equivalent reduced terms produce different elements when applied to a free base in a standard free construction.
	
	\begin{cor}\label{May19_3}  The tuple $\bar{a}$ is independent if and only if $t(\bar{a})\neq r(\bar{a})$ for all reduced terms $t,r$ such that $t\not\sim r$.
	\end{cor}
	\begin{proof} The direction from left to right follows from  Proposition~\ref{May19_1}(2). For the other direction, assume that whenever $t$, $r$ are reduced and $t \not\sim r$ we have $t(\bar{a})\neq r(\bar{a})$. Let $N$ be a Steiner quasigroup and let $f: \bar{a} \to N$ be a map. We extend $f$ to $\hat{f} : \langle \bar{a} \rangle \to N$ as follows: for $b \in \langle \bar{a} \rangle$, let $t$ be a reduced term such that $b=t(\bar{a})$. Then
		\[ \hat{f}(b):= t(f(\bar{a})) \,.\]
		Suppose $r$ is a reduced term such that $r(\bar{a})=t(\bar{a}) = b$. Then by hypothesis $r \sim t$. Hence $\hat{f}$ is well defined. We claim that $\hat{f}$ is a homomorphism. 
		For $b_1, b_2 \in \langle \bar{a} \rangle$, choose reduced terms $t_1$, $t_2$, $r$ such that $b_i=t_i(\bar{a})$ for $i=1,2$, 
		\[ \hat{f}(b_1b_2) = \hat{f}(t_1(\bar{a})t_2(\bar{a})) \, . \]
		If $t_1t_2$ is reduced, then $\hat{f}(t_1t_2(\bar{a}))= t_1(f(\bar{a})) t_2(f(\bar{a}))= \hat{f}(b_1)\hat{f}(b_2)$. 
		Otherwise, suppose $t_1 \sim t_2$. Then 
		$t_1(\bar{a})=t_2(\bar{a})=b_1=b_2$ and 
		\[t_1(f(\bar{a}))=t_2(f(\bar{a}))=\hat{f}(b_1)=\hat{f}(b_2) \,. \]
		The homomorphism property follows.
		The other possibility is $t_1 \sim t_2s$ for a reduced term~$s$. Let $c=s(\bar{a})$, so $b_1=b_2c$. Then
		$b_1b_2=c \,$
		and 
		\[\hat{f}(b_1)\hat{f}(b_2)= \hat{f}(b_2c)\hat{f}(b_2)= [t_2(f(\bar{a}))s(f(\bar{a}))]t_2(f(\bar{a}))=s(f(\bar{a}))=\hat{f}(b_1b_2) \, . \] 
				
	\end{proof}

	\begin{cor} \label{May19_5} A tuple $\bar{a}$ is independent if and only if all its finite subtuples are independent. \end{cor}
	\begin{proof} By Corollary~\ref{May19_3}, independence fails for a finite subtuple $a_{i_0}, \ldots, a_{i_n}$ of $\bar{a}$ on suitable reduced terms $t$ and $r$.
	\end{proof}
	
	\begin{cor}\label{May19_4} An ascending chain of independent sets is independent.  
	\end{cor}
	
	\begin{proof} Suppose $\langle A_i : i \in I \rangle$ is an ascending chain of subsets of a  Steiner quasigroup. Suppose $A=\bigcup_{i \in I} A_i$ is not independent and enumerate $A$ as $\bar{a}$. Then by Corollary~\ref{May19_5}, there is a finite $\bar{a}^\prime \subseteq \bar{a}$ that is not independent. Since $\bar{a}^\prime \subseteq A_i$ for some $i < \omega$, we have that $A_i$ is not independent.
	\end{proof}

	\begin{lemma}\label{May26_1} Let $\bar{a}$ be independent. If there are reduced terms $t(\bar{x},y)$ and $r(\bar{x},y)$ such that
		\begin{itemize}[topsep=0pt,itemsep=0pt]
			\item $t(\bar{a},b)=  r(\bar{a},b)$
			\item $t\not\sim r$
		\end{itemize}
		then there is a  reduced term $t^\prime(\bar{x},y)$ such that
		\begin{itemize}[topsep=0pt,itemsep=2pt]
			\item $t^\prime(\bar{a},b)\in\langle\bar{a}\rangle$
			\item $y$ appears in $t^\prime$.
		\end{itemize}
	\end{lemma}
	\begin{proof} Let $\bar{a}= a_1,\ldots,a_n$ and $\bar{x}= x_1,\ldots,x_n$. Note that $y$ appears in $t$ or in $r$, since otherwise  $t\sim r$  by Corollary~\ref{May19_3}.  We may assume that $y$ appears in $t$ and in $r$, because if, say, $y$ appears only in $t$, then $t^\prime= t$ gives the result. Moreover some $x_i$ appears in $t$  or in $r$, because otherwise $t=y=r$ and then $t\sim r$. Without loss of generality, $x_i$ appears in $t$. We argue by induction on $t$.
		
		The case where $t$ is a variable is impossible, because then $t=x_i =y $.
		
		Assume $t=t_1\cdot t_2$.  Without loss of generality, $x_i$ appears in $t_1$.  We have
		$$t_1(\bar{a},b)= r(\bar{a},b)\cdot t_2(\bar{a},b)$$
		If $t_1\sim r\cdot t_2$, then $t= t_1\cdot t_2\sim (r\cdot t_2)\cdot t_2$ and $t$ is not reduced. Hence $t_1\not\sim r\cdot t_2$. If $r\cdot t_2$ is reduced, we can apply the induction hypothesis. Therefore, we may assume $r\cdot t_2$ is not reduced. By Lemma~\ref{May12_3}\,(6), there are three cases.
		
		\underline{Case 1}: \ $r\sim t_2$.  In this situation we have
		$$t_1(\bar{a},b)= r(\bar{a},b)\cdot t_2(\bar{a},b)= t_2(\bar{a},b)\cdot t_2(\bar{a},b) =t_2(\bar{a},b)$$
		and $t_1\not\sim t_2$ (because otherwise $t$ is not reduced). Since $y$ appears in $t_1$ or in $t_2$, we may apply the induction hypothesis to get the result.
		
		\underline{Case 2}: \  $r\sim t_2\cdot s$  for some term $s$. Then $s$ is reduced and $t_1(\bar{a},b)= s(\bar{a},b)$. If $t_1\sim s$, then $t=t_1\cdot t_2 \sim s\cdot t_2 \sim r$, a contradiction. Hence, $t_1\not\sim s$. There are two subcases. Subcase 2.1 is that $y$ appears in $s$. In this situation, we can apply the induction hypothesis to the equality $t_1(\bar{a},b)= s(\bar{a},b)$.  Subcase 2.2 is that $y$  does not appear in $s$. Then  $t_1(\bar{a},b)\in\langle\bar{a}\rangle$. If, moreover, $y$ appears in $t_1(\bar{x},y)$, we are done. Otherwise, $t_1= t_1(\bar{x})$, $s = s(\bar{x})$ and  $t_1(\bar{a})= s(\bar{a})$. By Proposition~\ref{May19_1}, $t_1\sim s$. Then  $t= t_1\cdot t_2\sim s\cdot t_2\sim r$, contradicting $t\not\sim r$.
		
		\underline{Case 3}: \ $t_2\sim r\cdot s$ for some term $s$. Then, again, $t_1(\bar{a},b)= s(\bar{a},b)$ and $s$ is reduced. There are again two subcases, according to whether $y$ appears in $s$ or not. The first subcase is like subcase 2.1. In the second subcase we assume that $y$ does not appear in $s$. As in subcase 2.2, we see that $t_1\sim s$. Then $t= t_1\cdot t_2 \sim s\cdot (r\cdot s)$ and $t$ is not reduced, a contradiction.
	\end{proof}

	\begin{prop}\label{May26_2} Let $\bar{a}$  be independent. Then  $\bar{a}b$ is independent if and only if  for every reduced term $t(\bar{x},y)$ containing $y$, we have  $t(\bar{a},b)\not\in\langle \bar{a}\rangle$.
	\end{prop}
	\begin{proof} From left to right: assume $\bar{a}b$ is independent and  $t(\bar{a},b)\in\langle \bar{a}\rangle$  for some reduced term containing the variable $y$. Then $t(\bar{a},b)= r(\bar{a})$  for some reduced term $r(\bar{x})$. By Proposition~\ref{May19_1}, $t\sim r$. But this implies that $y$ appears in $r$, a contradiction.
		
		From right to left: assume $\bar{a}b$ is not independent. By Corollary~\ref{May19_3}, there are reduced terms $t(\bar{x},y)$ and $r(\bar{x},y)$ such that  $t(\bar{a},b)=r(\bar{a},b)$  and $t\not\sim r$. Since $\bar{a}$ is independent, the same corollary shows that $y$ appears in $t$ or in $r$. By Lemma~\ref{May26_1}, there is a reduced term $t^\prime(\bar{x},y)$ such that  $y$ appears in $t^\prime$ and   $t^\prime(\bar{a},b)\in\langle\bar{a}\rangle$, which contradicts the right hand side.
	\end{proof}
	
	\begin{cor}\label{May26_3} Let $\bar{a}b$  be independent, let $t(\bar{x},y)$ be a reduced term containing $y$ and let  $f:\langle\bar{a}b\rangle \rightarrow \langle\bar{a},t(\bar{a},b)\rangle$  be the surjective homomorphism induced by  $f(\bar{a})=\bar{a}$  and  $f(b)=t(\bar{a},b)$. The following are equivalent:
		\begin{enumerate}
			\item $f$ is one-to-one
			\item $\bar{a},t(\bar{a},b)$ is independent
			\item for every reduced term $r(\bar{x},y)$ containing $y$,  $r(\bar{a},t(\bar{a},b))\not\in\langle\bar{a}\rangle$.
		\end{enumerate}
	\end{cor}
	\begin{proof} 2 $\Leftrightarrow$ 3 follows from Proposition~\ref{May26_2}.
		
		1 $\Leftrightarrow$ 2:  if $f$ is one-to-one, then it is an isomorphism between $\langle\bar{a}b\rangle $  and $\langle\bar{a},t(\bar{a},b)\rangle$, which implies the independence of  $\bar{a},t(\bar{a},b)$. 
		
		For the other direction, assume  $\bar{a},t(\bar{a},b)$ is independent. Then  $g:\langle \bar{a},t(\bar{a},b)\rangle \rightarrow \langle \bar{a},b\rangle$  defined by  
		\[g(\bar{a})=\bar{a} \mbox{ and }g(t(\bar{a},b))= b \]
		 is a surjective homomorphism. Then  $g\circ f:\langle \bar{a},b\rangle \rightarrow \langle\bar{a},b\rangle$ is a homomorphism and it is the identity on $\bar{a}b$. By uniqueness of the induced homomorphism, $g\circ f$ is the identity, which implies that $f$ is one-to-one.
	\end{proof}
	
	\section{Endomorphisms}
	
We give a number of results that concern the syntax of terms of the form $t(\bar{x},y)$. When $\bar{a}b$ is an independent tuple in a Steiner quasigroup $M$, these results allow us to give conditions under which the mapping $\bar{a} \mapsto \bar{a}$, $b \mapsto t(\bar{a},b)$ induces an automorphism of $\langle \bar{a}b \rangle$, or an embedding or an isomorphism between the subquasigroups $\langle\bar{a}b \rangle$ and $\langle \bar{a}, \, t(\bar{a},b) \rangle$ of~$M$.
	
	\begin{df} A subterm $r$ of $t$ is {\bf single} if for any subterms  $r_1\cdot t_1$ and $r_2\cdot t_2$ of $t$ (or $r_1\cdot t_1$, $t_2\cdot r_2$,  or  $t_1\cdot r_1$, $r_2\cdot t_2$ or  $t_1\cdot r_1$, $t_2\cdot r_2$)  with  $r_1\sim r_2\sim r$  we have  $t_1\sim t_2$. Otherwise we call it {\bf double}.
	\end{df}
	
	\begin{lemma}\label{June9_1} Let $t= t(\bar{x},y)$ be reduced. Then  $y$ has an occurrence in $t$ such that every subterm of $t$ containing it is single if and only if  $t$ has only one occurrence of $y$.
	\end{lemma}
	\begin{proof} Assume $y$ has an occurrence in $t$ such that every subterm of $t$ containing it is single.  By looking at the syntactical labelled tree of construction of $t$ and at the branch starting with this occurrence of $y$, we see that there are terms $t_1,\ldots,t_n$  such that
		\begin{itemize}
			\item $y = t_1$
			\item $t= t_n$
			\item $t_{i+1}= t_i\cdot r_i$  or  $t_{i+1}= r_i\cdot t_i$ for some term $r_i$  if $i\geq 1$
			\item each $t_i$ is single.
		\end{itemize}
		We now claim that if $j\leq i$, then there is no subterm $r$ of $r_i$  such that $t_j\sim r$. This claim implies that $y$ does not appear in any $r_i$ and hence that $y$ has only one occurrence in $t$. 
		
		First consider the case $j=i$. On one hand, since $t$ is reduced, $t_i\not\sim r_i$; on the other hand, $r_i$ does not have a term of the form  $s_1\cdot s_2$  with $t_i\sim s_1$ or $t_i\sim s_2$,  because, since $t_i$ is single, this would imply $s_1\sim r_i$ or $s_2\sim r_i$, which is impossible since  $s_1,s_2$  are proper subterms of $r_i$.
		
		We finish the proof by showing that if there is some subterm $s$ of $r_i$ such that $s\sim t_j$ for some $j < i$, then there is also a subterm $s^\prime$ of $r_i$ such that $s^\prime\sim t_{j+1}$:  if $t_j\sim r_i$, then $r_j\sim t_i$, which is impossible; and if  $t_j\sim s$ for a proper subterm $s$ of $r_i$, then for some term $s^\prime$, $s\cdot s^\prime$  (or $s^\prime\cdot s$) is a subterm of $r_i$ and then  $s^\prime\sim r_j$ and $t_{j+1}\sim t_j\cdot r_j \sim s\cdot s^\prime$, which is a subterm of $r_i$.
		
		For the other direction, assume that $t$ has only one occurrence of $y$. We claim that every subterm $r$ of $t$ containing the occurrence of $y$ must be single. For this, assume $r_1\cdot t_1$ and $r_2\cdot t_2$ are subterms of $t$  with  $r_1\sim r_2\sim r$. Then $r_1$ and $r_2$ contain $y$. We show that $t_1\sim t_2$.  Since each of $r,r_1,r_2$ must be a subterm of the other two, it follows that $r=r_1=r_2$. Then $r\cdot t_1$ is a subterm of $r\cdot t_2$ (or viceversa), which is impossible unless $t_1=t_2$.   Similarly if the subterms of $t$ are $r_1\cdot t_1$, $t_2\cdot r_2$,  or  $t_1\cdot r_1$, $r_2\cdot t_2$ or  $t_1\cdot r_1$, $t_2\cdot r_2$.
\end{proof}
	
	\begin{lemma}\label{June9_2} If $t= t(\bar{x},y)$ has only one occurrence of $y$, then there is a term $r(\bar{x},z)$  such that the equations  $t(\bar{x},y)= z$   and   $r(\bar{x},z)= y$  are equivalent in every Steiner quasigroup.
	\end{lemma}
	\begin{proof} Induction on $t$.  If $t$ is a variable, $t= y$  and we can take  $r= z$.  If $t=t_1\cdot t_2$, without loss of generality, $y$ appears in $t_1$ and not in $t_2$. Then $t(\bar{x},y)=z$ is equivalent to  $t_1(\bar{x},y)= t_2(\bar{x})\cdot z$. By induction hypothesis, there is a term $r(\bar{x},u)$  for which  $t_1(\bar{x},y)= u$  is equivalent to  $r(\bar{x},u)= y$.  Then  $t_1(\bar{x},y)= t_2(\bar{x})\cdot z$  is equivalent to  $r(\bar{x},t_2(\bar{x})\cdot z)= y$.
	\end{proof}
	
	\begin{prop}\label{June9_3} If $t= t(\bar{x},y)$ has only one occurrence of $y$ and $\bar{a}b$ is independent, where $\bar{a}=a_1,\ldots,a_n$, then the mapping induced by
		\begin{itemize}
			\item $a_i\mapsto a_i$  for $i=1,\ldots,n$
			\item $b\mapsto t(\bar{a},b)$
		\end{itemize}
		is an automorphism of  $\langle \bar{a}b\rangle$.
	\end{prop}
	\begin{proof} Since $\bar{a}b$ is independent, there is a homomorphism $f:\langle\bar{a}b\rangle\rightarrow \langle\bar{a}b\rangle$  such that  $f(\bar{a}) = \bar{a}$    and  $f(b)= t(\bar{a},b)$.  We claim that there is another homomorphism $g:\langle\bar{a}b\rangle\rightarrow \langle\bar{a}b\rangle$  such that  $g(\bar{a}) = \bar{a}$   and  $g(t(\bar{a},b))= b$. Then  $g\circ f:\langle\bar{a}b\rangle \rightarrow \langle\bar{a}b\rangle$ is a homomorphism  such that  $g\circ f(\bar{a})= \bar{a}$    and $g\circ f(b)=b$. Hence $g\circ f$ is the identity on $\langle\bar{a}b\rangle$, which implies that $f$ is one-to-one. Lemma~\ref{June9_2}  shows that $b\in\langle \bar{a}, t(\bar{a},b)\rangle$, which is the range of $f$. Hence $f$ is surjective.
		
		We prove the claim by induction on $t$.  If  $t$ is a single variable, it must be  $t= y$ and then $f(b)=b$  and we take $g=f$.  Assume $t= t_1\cdot t_2$. Without loss of generality, $y$ occurs in $t_1$ and not in $t_2$.  The induction hypothesis gives a homomorphism $h:\langle\bar{a}b\rangle\rightarrow\langle \bar{a}b\rangle$ such that  $h(\bar{a})= \bar{a}$ and $h(t_1(\bar{a},b))= b$. Let $a= t_2(\bar{a})$ and note that  $h(a)=a$  and $h(t(\bar{a},b))= h(t_1(\bar{a},b))\cdot h(t_2(\bar{a}))= b\cdot a$. Let $j:\langle\bar{a}b\rangle\rightarrow\langle \bar{a}b\rangle$ be the homomorphism induced by  $j(\bar{a})=  \bar{a}$   and $j(b) =b\cdot a$. Then  $j(b\cdot a)= (b\cdot a)\cdot a =b$ and hence  $j\circ h(t(\bar{a},b))= j(b\cdot a) = b$   and  $j\circ f(\bar{a})= \bar{a}$. Therefore $g= j\circ h$ satisfies the requirements for $t$.
	\end{proof}
	
	\begin{cor}\label{June9_4}  If $t= t(\bar{x},y)$  is reduced and  $y$ has an occurrence in $t$ such that every subterm of $t$ containing this occurrence is single,
		and $\bar{a}b$ is independent, where $\bar{a}=a_1,\ldots,a_n$, then the mapping induced by
		\begin{itemize}
			\item $a_i\mapsto a_i$  for $i=1,\ldots,n$
			\item $b\mapsto t(\bar{a},b)$
		\end{itemize}
		is an automorphism of  $\langle \bar{a}b\rangle$.
	\end{cor}
	\begin{proof} By Lemma~\ref{June9_1} and Proposition~\ref{June9_3}.
	\end{proof}

\begin{lemma}\label{March_8_23_1} Assume $\bar{x}= \langle x_i : i \in I\rangle$ is a tuple of variables and  $y,z,x_i$ are pairwise distinct. Let  $t_i=t_i(\bar{x},z)$  for $i=1,2$, let $r=r(\bar{x},y)$  and assume $y$ occurs in $r$. We write  $t_i^\prime= t_i(\bar{x},r(\bar{x},y))$  for the substitution, and similarly for other terms.  If  $t_1^\prime\sim t_2^\prime$, then  $t_1\sim t_2$.
\end{lemma}
\begin{proof} By induction on $t_1$.

(i) Let $t_1= x_i$.  Then  $t_1^\prime= x_i\sim t_2^\prime$, which implies $t_2^\prime= x_i$. Since  $y$ occurs in $r$,  $t_2= x_i$ and hence $t_1\sim t_2$.

(ii) Let  $t_1= z$.  Then $t_1^\prime = r\sim t_2^\prime$, which implies that $y$ occurs in $t_2^\prime$ and, therefore, that $z$ occurs in $t_2$.  If $t_2\neq z$, then $\len(t_2^\prime)>\len(r)=\len(t_1^\prime)$, which is impossible if $t_1^\prime\sim t_2^\prime$. Thus, $t_2= z$ and then $t_1\sim t_2$.

(iii) Let  $t_1= s_1\cdot s_2$  for some terms, $s_1,s_2$.  Then $t_1^\prime= s_1^\prime\cdot s_2^\prime\sim t_2^\prime$. Note that $t_2\neq x_i$, since otherwise $t_2^\prime= t_2=x_i\not\sim s_1^\prime\cdot s_2^\prime$. In case $t_2= z$, we get $t_2^\prime=r$, but since $z$ is in  $t_1$ (because $y$ occurs in $t_1^\prime$), then $\len(t_1^\prime)>\len(r)=\len(t_2^\prime)$, contradicting $t_1^\prime\sim t_2^\prime$. Hence, $t_2\neq z$. Therefore, $t_2= u_1\cdot u_2$  for some terms $u_1,u_2$.  Therefore, $t_2^\prime=u_1^\prime\cdot u_2^\prime\sim s_1^\prime\cdot s_2^\prime$. Without loss of generality,  $u_1^\prime\sim s_1^\prime$ and  $u_2^\prime\sim s_2^\prime$. By inductive hypothesis, $u_1\sim s_1$ and $u_2\sim s_2$.  Then $t_1= s_1\cdot s_2\sim u_1\cdot u_2 = t_2$.
\end{proof}

\begin{lemma}\label{March_8_23_2} 
Assume $\bar{x}= \langle x_i : i \in I\rangle$ is a tuple of variables and  $y,z,x_i$ are pairwise distinct. Let  $t=t(\bar{x},z)$,  let $r=r(\bar{x},y)$  and assume $y$ occurs in $r$. As above, we write  $t^\prime= t(\bar{x},r(\bar{x},y))$  for the substitution. Assume $r=r_1\cdot r_2$ and $y$ occurs in each $r_i$.  If $t$ and $r$ are reduced, then $t^\prime$ is reduced.
\end{lemma}
\begin{proof} By  induction on $t$.

(i) Let $t=x_i$.  Then  $t^\prime = t$  is reduced.

(ii) Let $t=z$. Then $t^\prime = r$ is reduced.

(iii) Let $t= t_1\cdot t_2$  for some terms $t_1,t_2$. By inductive hypothesis, each  $t_i^\prime =t_i(\bar{x},r(\bar{x},y))$  is reduced.  If $t^\prime =t_1^\prime\cdot t_2^\prime$ is not reduced, then either  $t_1^\prime\sim t_2^\prime$  or  $t_1^\prime\sim t_2^\prime\cdot s$  for some term $s$  or  $t_2^\prime\sim t_1^\prime\cdot s$  for some term $s$.  If $t_1^\prime\sim t_2^\prime$, we can apply Lemma~\ref{March_8_23_1} and obtain $t_1\sim t_2$, which implies that $t$ is not reduced.  Now assume, without loss of generality, that $t_1^\prime\sim t_2^\prime\cdot s$  for some term $s$.  We consider three subcases:

\textit{Subcase (a)}: \ $t_1=x_i$. Then  $t_1^\prime = x_i\sim t_2^\prime\cdot s$, which is impossible by length considerations.

\textit{Subcase (b)}: \ $t_1 = z$.  Then $t_1^\prime=r_1\cdot r_2\sim t_2^\prime\cdot s$. The case  $r_1\sim t_2^\prime$  is not possible, since  it implies that $z$ occurs in $t_2$ and then $\len(t_2^\prime)\geq \len (r)>\len(r_1)$. Hence  $r_1\sim s$ and $r_2\sim t_2^\prime$.  But again this is impossible, because it implies that $z$ occurs in $t_2$ and $\len(t_2^\prime)\geq \len(r)>\len(r_2)$.

\textit{Subcase (c)}: \  $t_1= s_1\cdot s_2$  for some terms $s_1,s_2$. Then $t_1^\prime=s_1^\prime\cdot s_2^\prime\sim t_2^\prime\cdot s$. Without loss of generality, $s_1^\prime\sim t_2^\prime$ and $s_2^\prime\sim s$. By  Lemma~\ref{March_8_23_1}, $s_1\sim t_2$. Then $t_1\sim t_2\cdot s_2$  and then  $t\sim (t_2\cdot s_2)\cdot t_2$  is not reduced.
\end{proof}

\bigskip

\begin{prop}\label{finalprop} Let $\bar{a},b$ be independent, let $r(\bar{x},y)$ be a reduced term with at least two occurrences of $y$, let $c=r(\bar{a},b)$  and let $f:\langle \bar{a},b\rangle\rightarrow \langle\bar{a},c\rangle$ be the surjective homomorphism induced by $f(\bar{a})=\bar{a}$  and $f(b)=c$. Then  $f$ is an isomorphism between $\langle \bar{a},b\rangle$ and $\langle\bar{a},c\rangle$  and  $b\not \in\langle\bar{a},c\rangle$. In fact, for every term $t(\bar{x},y)$ having exactly one occurrence of $y$,  $t(\bar{a},b)\not \in\langle \bar{a},c\rangle$.
\end{prop}
\begin{proof} Consider first the case  $r= r_1\cdot r_2$, where each $r_i$ contains an occurrence of $y$.  In this case, by Lemma~\ref{March_8_23_2}, for every reduced term $t(\bar{x},z)$,  the term $t^\prime(\bar{x},y)= t(\bar{x},r(\bar{x},y))$ is reduced. 

We prove first that $f$ is one-to-one.  Assume $f(c_1)= f(c_2)$. For each $i$ there is some reduced term $t_i(\bar{x},z)$  such that  $c_i= t_i(\bar{a},b)$.  Since the terms $t_1^\prime,t_2^\prime$ are reduced and  $t_1^\prime(\bar{a},b)= f(c_1)= f(c_2)=t_2^\prime(\bar{a},b)$, by  Proposition~\ref{May19_1},  $t_1^\prime\sim t_2^\prime$.  Then  $t_1\sim t_2$ by Lemma~\ref{March_8_23_1}. Then  $c_1=t_1(\bar{a},b)=t_2(\bar{a},b)= c_2$.

Now we prove that $b\not\in \langle\bar{a},c\rangle$. Assume $b\in\langle\bar{a},c\rangle$. Then there is a reduced term $t(\bar{x},z)$ such that $b= t(\bar{a},c)$. Note that $z$ occurs in $t$, since otherwise $b\in\langle \bar{a}\rangle$.  By Lemma~\ref{March_8_23_2},  $t^\prime(\bar{x},y)= t(\bar{x},r(\bar{x},y))$ is reduced. Clearly, $b= t^\prime(\bar{a},b)$.  If $k=\rk(t^\prime)$, Proposition~\ref{May19_1} gives  $t^\prime(\bar{a},b)\in S_k(\bar{a},b)\smallsetminus S_{<k}(\bar{a},b)$.  But $k\geq \rk(r) >0$ and  $b\in S_0(\bar{a},b)$, a contradiction.

Now we consider the general case of a term $r(\bar{x},y)$  with at least two occurrences of $y$.  We claim that $f$ can be written as a composition $f=g\circ h$, where

(i)  $h$  is an isomorphism between $\langle \bar{a},b\rangle$ and $\langle \bar{a},d\rangle$,  where $d= s(\bar{a},b)$  for a reduced term  $s(\bar{x},y)$  which is a product  $s=s_1\cdot s_2$ of terms $s_i$ containing $y$,   and  $h$ is induced by  $h(\bar{a})= \bar{a}$  and $h(b)=s(\bar{a},b)=d$,  and

(ii) $g$ is an automorphism of  $\langle \bar{a},d\rangle =\langle\bar{a},c\rangle$ and  $c= t(\bar{a},d)$  for a term $t(\bar{x},y)$  with exactly one occurrence of $y$  and  $g$ is induced by $g(\bar{a})=\bar{a}$  and  $g(d)=t(\bar{a},d)= c$.

We prove this by induction on $r$. Clearly, $r$ is not a variable, hence $r=r_1\cdot r_2$  for some reduced terms $r_i= r_i(\bar{x},y)$.  If each $r_i$ contains an occurrence of $y$, we are in the case we have discussed at the beginning and we can take $f=h$ and take $g$ to be the identity on  $\langle \bar{a},c\rangle$. If one of the terms $r_i$, say $r_1$, does not contain $y$, then $r_1= r_1(\bar{x})$  and $r_2(\bar{x},y)$ contains at least two occurrences of $y$. By inductive hypothesis for $r_2$,   we obtain an isomorphism $f^\ast:\langle \bar{a},b\rangle\rightarrow \langle\bar{a},c^\ast\rangle$, induced by  $f^\ast(\bar{a})=\bar{a}$  and $f^\ast(b)=c^\ast= r_2(\bar{a},b)$ which can be decomposed as  $f^\ast = g^\ast\circ h$, where $h:\langle\bar{a},b\rangle \rightarrow \langle\bar{a},d\rangle =\langle \bar{a},c^\ast\rangle$ is induced by $h(\bar{a})=\bar{a}$ and  $h(b)= s(\bar{a},b)=d$, where $s =s_1\cdot s_2$ is reduced and each $s_i $  contains an occurrence of $y$,  and  $g^\ast$ is an automorphism of $\langle\bar{a},d\rangle$ induced by $g^\ast(\bar{a})=\bar{a}$  and  $g^\ast(d)=t^\ast(\bar{a},d)= c^\ast$, where $t^\ast(\bar{x},y)$  contains a unique occurrence of $y$. Now  $t= r_1\cdot t^\ast$  is a term with a unique occurrence of $y$  and  $t(\bar{a},d)= r_1(\bar{a})\cdot t^\ast(\bar{a},d)= r_1(\bar{a})\cdot c^\ast= r_1^\ast(\bar{a})\cdot r_2(\bar{a},b)= r(\bar{a},b)= c$. Hence, there is an automorphism $g$ of $\langle\bar{a},d\rangle$  induced by  $g(\bar{a})=\bar{a}$  and  $g(d)= t(\bar{a},d)= c$  and $f= g\circ h$  is the desired decomposition.
\end{proof}

	\begin{ex}\label{ex1}
		Let $\bar{a}=a_1a_2$ and let $\bar{a}b$ be independent. Let  $f:\langle\bar{a}b\rangle \rightarrow \langle \bar{a}b\rangle$ be the homomorphism induced  by 
		\begin{itemize}
			\item $a_i\mapsto a_i$ for $i=1,2$
			\item $b\mapsto (a_1 \cdot b)\cdot  (a_2\cdot  b)$.
		\end{itemize}
		Then Proposition~\ref{finalprop} with $r(\bar{x}, y)=(x_1\cdot y)\cdot (x_2\cdot y)$ gives $b\not\in\langle\bar{a},(a_1\cdot  b)\cdot  (a_2\cdot  b)\rangle$, so $f$ is not surjective. However, $f$ induces an isomorphism between $\langle \bar{a},b\rangle$ and $\langle \bar{a}, (a_1\cdot b)\cdot (a_2\cdot b)\rangle \subseteq \langle \bar{a},b\rangle$, and so it is an embedding of $\langle\bar{a},b\rangle$ into itself.
		
For a model theoretic application of this example,  recall that an element of a structure $M$ is \textit{definable} over a subset $A\subseteq M$ if it satisfies a first-order formula with parameters in $A$ that has a unique solution in $M$, and it is \textit{algebraic} over A if the formula has finitely many solutions. Then the algebraic (resp. definable) closure of $A$ is the set of all elements of $M$ that are algebraic (resp. definable) over $A$. 

Let $c=(a_1\cdot b)\cdot (a_2\cdot b)$. Then  the algebraic closure of $\{a_1,a_2,c \}$ properly contains the generated Steiner quasigroup $\langle \bar{a},c \rangle$.  In fact  it is easy to see that any unconfined Steiner quasigroup $M$ extending $\langle \bar{a}\rangle$ satisfies 
$$ \forall y_1 y_2 \,( (a_1\cdot y_1)\cdot (a_2\cdot y_1)=(a_1\cdot y_2)\cdot (a_2\cdot y_2) \rightarrow y_1=y_2).$$
This is because two distinct witnesses for $y$ in $(a_1\cdot y)\cdot (a_2\cdot y)$ would introduce a confined configuration.
It follows that $b$ is definable over $\bar{a}c$, hence in particular $b$ is algebraic over $\bar{a}$, but $b \notin \langle \bar{a},c \rangle$.
\end{ex}

	\begin{cor}\label{July18_1} Let  $\bar{a}= a_1,\ldots,a_n$, and assume $\bar{a},b$ is independent and $c\in \langle \bar{a}b\rangle \smallsetminus \langle \bar{a}\rangle$. The following are equivalent:
	\begin{enumerate}
		\item $b\in\langle\bar{a},c\rangle$
		\item $c= t(\bar{a},b)$ for some reduced term $t(\bar{x},y)$ with a unique occurrence of $y$
		\item  the homomorphism $f:\langle \bar{a},b\rangle \rightarrow  \langle \bar{a},b\rangle$  defined by $f(\bar{a})= \bar{a}$  and  $f(b)= c$ is surjective
		\item  the homomorphism $f:\langle \bar{a},b\rangle \rightarrow  \langle \bar{a},b\rangle$  defined by $f(\bar{a})= \bar{a}$  and  $f(b)= c$ is an automorphism of $\langle \bar{a},b\rangle$.
	\end{enumerate}
\end{cor}
\begin{proof}
It is clear  that 1 $\Leftrightarrow$ 3  and that 4 $\Rightarrow$ 3.  By Proposition~\ref{June9_3}, we know that  2 $\Rightarrow$ 4.
Finally,    3 $\Rightarrow$ 2 follows from Proposition~\ref{finalprop}.

\end{proof}

\section{Tame automorphisms}

A {\bf Steiner loop} is a triple $(M,\cdot,1)$ where $\cdot$ is a binary operation on $M$, $1\in M$ and
 \begin{enumerate}
	\item $\forall x \, (x \cdot y= y\cdot x)$
	\item $\forall x \, (x\cdot x =1)$
	\item $\forall x \, (x\cdot (x\cdot y)=y)$.
\end{enumerate}
Note that this implies  $\forall x \, (x\cdot 1 =x)$.  A Steiner loop can be obtained from a Steiner quasigroup by adding the element $1$ and redefining $x\cdot x$ as $x\cdot x= 1$. Conversely, every Steiner loop produces a Steiner quasigroup by eliminating  $1$ and redefining $x\cdot x$ as $x\cdot x=x$. 

 Every embedding between Steiner quasigroups induces a corresponding embedding between the associated Steiner loops, and conversely. Steiner quasigroups and Steiner loops have, essentially, the same automorphism groups. If $f$ is a homomorphism between two Steiner quasigroups, $f$ can be extended to a homomorphism of the associated Steiner loops by setting  $f(1)=1$. But a homomorphism $f$ of Steiner loops does not define a homomorphism of the corresponding Steiner quasigroups, since there might be elements $a\neq 1$ with $f(a)=1$.

There is a particular result on automorphisms of Steiner loops in~\cite{grishkovetal} that is relevant for us and can be succesfully translated to our setting. It implies that every automorphims of a finitely generated free Steiner quasigroup can be obtained as a finite product of automorphisms of the kind studied in the previous section and described in Corollary~\ref{July18_1}, and even of a smaller class of automorphisms called \textit{elementary}. This result is the content of this section.

Let $\bar{a}= a_1,\ldots,a_n$  be independent elements of a Steiner quasigroup. An endomorphism $f:\langle\bar{a}\rangle\rightarrow \langle\bar{a}\rangle$ is {\bf elementary} if for some $i$ and some term $t(\bar{x})$ not containing $x_i$, $f$ is induced  by the mapping  $f(a_i)= a_i\cdot t(\bar{a})$  and $f(a_j)= a_j$ for $j\neq i$. Notice that, by  Proposition~\ref{June9_3}, every elementary endomorphism is an automorphism.  The automorphisms generated by the elementary automorphisms are called {\bf tame}. If $f$ is elementary, then $f=f^{-1}$, and hence every tame automorphism of $\langle \bar{a}\rangle$ is of the form $f_1\circ f_2\circ \ldots \circ f_m$  for some elementary automorphisms $f_i$. We show that every automorphism is tame, a proposition proven for Steiner loops in~\cite{grishkovetal}.

Let  $f:\langle \bar{a}\rangle\rightarrow\langle \bar{a}\rangle $ be an endomorphism induced by $a_i\mapsto t_i(\bar{a})$, where every $t_i(\bar{x})$ is a reduced term. We say that $f$ {\bf preserves reduced terms} if for every reduced term $t(\bar{x})$, the term  $t(t_1(\bar{x}),\ldots,t_n(\bar{x}))$ is reduced. We say that the tuple $t_1(\bar{x}),\ldots,t_n(\bar{x})$ of reduced terms is {\bf irreducible} if  there is no $i$ for which there are reduced terms $r(\bar{x}), s(\bar{x})$  such that  $t_i\sim r(t_1(\bar{x}),\ldots,t_n(\bar{x}))\cdot s(\bar{x})$. Notice that if there is such $i$, then $x_i$ does not occur in $r(\bar{x})$, since otherwise $\len(t_i)<\len(r(t_1(\bar{x}),\ldots,t_n(\bar{x}))\cdot s(\bar{x}))$.

\begin{lemma}\label{loops1} Let  $ \bar{a}=a_1,\ldots,a_n$ be independent elements of a Steiner quasigroup  and let $f:\langle \bar{a}\rangle\rightarrow\langle \bar{a}\rangle $ be an endomorphism induced by $a_i\mapsto t_i(\bar{a})$, where every $t_i(\bar{x})$ is a reduced term.
	\begin{enumerate} 
\item	If $f$ is an embedding, then $f$ preserves reduced terms if and only if the tuple of terms $t_1(\bar{x}),\ldots,t_n(\bar{x})$ is irreducible.
\item If $f$ is surjective, then $f$ preserves reduced terms  if and only if  every term $t_i$ is a variable.
\end{enumerate}
\end{lemma}
\begin{proof} 1. Let $f$ be defined by the reduced terms  $t_1,\ldots,t_n$  and let $r(\bar{x}), s(\bar{x})$ be reduced terms  such that $t_i\sim r(t_1(\bar{x}),\ldots,t_n(\bar{x}))\cdot s(\bar{x})$. Then  $x_i$ does not occur in $r$ and, therefore,  $u(\bar{x})=x_i\cdot r$ is reduced.   But  $u(t_1,\ldots,t_n)= t_i\cdot r(t_1,\ldots,t_n)\sim (r(t_1,\ldots,t_n)\cdot s )\cdot r(t_1,\ldots,t_n)$ is not reduced. Hence, $f$ does not preserve reduced terms.  For the other direction, we assume $t_1(\bar{x}),\ldots,t_n(\bar{x})$ is irreducible and we prove by induction on $r$ that for each reduced term $r(\bar{x})$,  $r(t_1(\bar{x}),\ldots,t_n(\bar{x}))$ is reduced. Let $b_i=t_i(\bar{a})$  and let $\bar{b}=b_1,\ldots,b_n$. Since $f$ is an embedding, $b_1,\ldots,b_n$ are independent.
	
	(i) Let $r=x_i$. Then $r(t_1,\ldots,t_n)=t_i$ is reduced.
	
	(ii) Let $r=r_1\cdot r_2$. Let $\bar{t}=t_1,\ldots,t_n$. By inductive hypothesis, $r_i(\bar{t})$ is reduced for $i=1,2$. If  $r(\bar{t})$ is not reduced, then,  by Lemma~\ref{May12_3}, $r_1(\bar{t})\sim r_2(\bar{t})$ or  $r_1(\bar{t})\sim r_2(\bar{t})\cdot s(\bar{x})$ or $r_2(\bar{t})\sim r_1(\bar{t})\cdot s(\bar{x})$  for some term $s(\bar{x})$. In the first case, $r_1(\bar{t}(\bar{a}))= r_2(\bar{t}(\bar{a}))$, that is, $r_1(\bar{b})= r_2(\bar{b})$ and, by Proposition~\ref{May19_1}, $r_1\sim r_2$, a contradiction, since $r$ is reduced. The second and third cases are symmetric. Without loss of generality, assume $r_1(\bar{t})\sim r_2(\bar{t})\cdot s(\bar{x})$ for some term $s(\bar{x})$. If $r_1=x_i$, then  $r_1(\bar{t})=t_i$  and $t_i\sim r_2(\bar{t})\cdot s(\bar{x})$, contradicting the irreducibility of $t_1(\bar{x}),\ldots,t_n(\bar{x})$. Hence $r_1$ is a product, say $r_1= s_1\cdot s_2$. Then $s_1(\bar{t})\cdot s_2(\bar{t})\sim r_2(\bar{t})\cdot s(\bar{x})$ and by Lemma~\ref{May12_3}, $s(\bar{x})\sim s_1(\bar{t})$ or $s(\bar{x})\sim s_2(\bar{t})$. Without loss of generality, $s(\bar{x})\sim s_1(\bar{t})$. Then $r_1(\bar{b})= r_2(\bar{b})\cdot s_1(\bar{b})$. We  check that $r_2(\bar{x})\cdot s_1(\bar{x})$ is reduced. If it is not, then $r_2\sim s_1$ or  $r_2\sim s_1\cdot s_3$ or $s_1\sim r_2\cdot s_3$ for some term $s_3(\bar{x})$. If $r_2\sim s_1$, then $r_1(\bar{b})=r_2(\bar{b})$ and, by Proposition~\ref{May19_1},  $r_1\sim r_2$, a contradiction since $r$ is reduced. In the case 
	$r_2\sim s_1\cdot s_3$, we have that $r_2(\bar{t})\sim s_1(\bar{t})\cdot s_3(\bar{t})$ and then $r_1(\bar{t})\sim r_2(\bar{t})\cdot s(\bar{x}) \sim (s_1(\bar{t})\cdot s_3(\bar{t}))\cdot s_1(\bar{t})$, a contradiction since $r_1(\bar{t})$ is reduced. In the case $s_1\sim r_2\cdot s_3$ we get a similar contradiction, since then $s_1(\bar{t})\sim r_2(\bar{t})\cdot s_3(\bar{t})$ and, therefore, $r_1(\bar{t})\sim r_2(\bar{t})\cdot (r_2(\bar{t})\cdot s_3(\bar{t}))$. We have shown that $r_2(\bar{x})\cdot s_1(\bar{x})$ is reduced. Since $\bar{b}$ is independent and $r_1(\bar{b})= r_2(\bar{b})\cdot s_1(\bar{b})$, this implies that $r_1(\bar{x})\sim r_2(\bar{x})\cdot s_1(\bar{x})$, which is impossible since $r(x)=r_1(\bar{x})\cdot r_2(\bar{x})$ is reduced.
	
	2. Assume $f$ is surjective and preserves reduced terms. Let $b_i=t_i(\bar{a})$  and let $\bar{b}=b_1,\ldots,b_n$.  By Lemma~\ref{May12_4}, for every $i$ there is a reduced term $r_i(\bar{x})$ such that $a_i=r_i(\bar{b})$. Then $a_i= r_i(\bar{t}(\bar{a}))$. Since $r_i(\bar{t})$ is reduced and $\bar{a}$ is independent $x_i\sim r_i(\bar{t})$.  But this is only possible if $r_i=x_j$ for some $j$ and  $t_j=x_i$.  For the other direction, notice that if every $t_i$ is a variable and $f$ is surjective, then $f$ is an automorphims induced by a permutation of $\{a_1,\ldots,a_n\}$. If $r(\bar{x})$ is reduced, then $r(\bar{t})$ is obtained from $r$ by a permutation of its variables $\bar{x}=x_1,\ldots,x_n$. For any term $s(\bar{x})$, let $s^\ast$ be the corresponding substitution by the same permutation of variables. An induction on $s$ shows easily that $s^\ast\sim u^\ast$ implies $s\sim u$. Using this, a straightforward induction on $s$ shows that if $s$ is reduced, $s^\ast$ is also reduced. Hence $r(\bar{t}) =r^\ast$ is reduced.
\end{proof}

\begin{prop}[Grishkov et al.~\cite{grishkovetal}]\label{loops2}
	Every automorphism of a finitely generated free Steiner quasigroup is tame.
\end{prop}
\begin{proof}  Let  $ \bar{a}=a_1,\ldots,a_n$ be independent elements of a Steiner quasigroup  and let $f$ be an automorphism of $\langle\bar{a}\rangle$ induced by the tuple of reduced terms $\bar{t}=t_1(\bar{x}),\ldots,t_n(\bar{x})$.  Let $b_i=f(a_i)=t_i(\bar{a})$ and $\bar{b}=b_1,\ldots,b_n$.  We prove by induction on $\len(\bar{t})=\len(t_1)+\ldots +\len(t_n)$ that $f$ is tame. The initial case is $\len(\bar{t})=n$. This means that each $t_i$ is a different variable and $f$ is induced by the corresponding permutation of $\{x_1,\ldots,x_n\}$. Since every permutation is a product of transpositions, it is enough to consider the case of a transposition $(i,j)$. Hence, $f(a_i)=a_j$, $f(a_j)=a_i$ and $f(a_k)=a_k$ if $k\neq i,j$. Let $g,h$ be the elementary automorphisms defined by $ g(a_i)=a_i\cdot a_j$ and $h(a_j)=a_j\cdot a_i$ (fixing all other elements of $\bar{a}$ in each case). It is easy to check that $f= g\circ h\circ g$ and, hence, $f$ is tame. Now consider the case $\len(\bar{t})>n$.  Then  some $t_i$ is not a variable and by  Lemma~\ref{loops1}, $t_1,\ldots,t_n$ are not irreducible. 
Hence, for some $i$  there are reduced terms $r(\bar{x}), s(\bar{x})$  such that  $t_i\sim r(t_1(\bar{x}),\ldots,t_n(\bar{x}))\cdot s(\bar{x})$, and it follows that $x_i$ does not occur in $r(\bar{x})$. The term $t_i\cdot r(\bar{t})$ is not reduced, but there is a reduced term $r_i(\bar{x})$ such that $r_i(\bar{a})= t_i(\bar{a})\cdot r(t_1(\bar{a}),\ldots,t_n(\bar{a}))$. Then $\len(r_i)\leq \len(s(\bar{x}))<\len(t_i)$ and we can apply the induction hypothesis to the automorphism $g$ induced by the terms $t_1,\ldots,t_{i-1},r_i,t_{i+1},\ldots,t_n$. It is, indeed, an automorphism, since $g=f\circ h$, where $h$ is the elementary automorphism where $x_i\mapsto x_i\cdot r(\bar{x})$ (and $x_i$ does not occur in $r(\bar{x})$).  Notice that  $g\circ h= f\circ h\circ h= f$, since $h$ is elementary and then $h\circ h$ is the identity. Since, by induction hypothesis, $g$ is tame, it is clear that $f$ is tame.
	\end{proof}

\section{A conjecture}

Example~\ref{ex1} shows how our results are relevant to model theoretic algebraic closure, an important tool in understanding the model theory of free Steiner triple systems. We conjecture that this example is an instance of a more general connection between the dependence of an element $b$ on an independent tuple $\bar{a}$ and definability and algebraicity of $b$ over~$\langle \bar{a}b \rangle$.

\begin{df}
	Let $t=t(\bar{x})$ and $r=r(\bar{x})$ be terms. We write $t\equiv r$ if in every Steiner quasigroup  $M$, for every tuple $\bar{a}\in M$,  $t(\bar{a})= r(\bar{a})$.
\end{df}

\begin{prop}
	Assume $\bar{x}=x_1,\ldots,x_n$  and  $t(\bar{x}), r(\bar{x})$ are terms.
	\begin{enumerate}
		\item If $t\sim r$, then $t\equiv r$.
		\item If $t,r$ are reduced, the following are equivalent:
		\begin{enumerate}
			\item $t\equiv r$.
			\item If $\bar{a}=a_1,\ldots,a_n$ is a base of the Steiner quasigroup $M$, then  $t(\bar{a})=r(\bar{a})$.
			\item $t\sim r$.
		\end{enumerate}
	\end{enumerate}
\end{prop}
\begin{proof}
	1 is clear.  2 follows from  Proposition~\ref{May19_1}.
\end{proof}

Every term can be transformed into a reduced term just by replacing subterms of the form $t\cdot t$ by $t$  and subterms of the form $t\cdot(r\cdot t)$ (and variants like $(r\cdot t)\cdot t$) by $r$. The following definition gives an explicit algorithm to find the result of these operations.

\begin{df}
	For any term $t$ we define its reduced form  $t^\mathrm{red}$ as follows.
	\begin{enumerate}
		\item If $t=x$ is a variable,  $t^\mathrm{red}= x$.
		\item Assume $t=t_1\cdot t_2$.\begin{enumerate}
		\item If $t_1^\mathrm{red}\sim t_2^\mathrm{red}$, then  $t^\mathrm{red}=t_1^\mathrm{red}$.
		\item If  $t_1^\mathrm{red}\not\sim t_2^\mathrm{red}$  and  $t_1^\mathrm{red}\sim r\cdot t_2^\mathrm{red}$, then $t^\mathrm{red} =r$.
		\item If $t_1^\mathrm{red}\not\sim t_2^\mathrm{red}$  and  $t_2^\mathrm{red}\sim r\cdot t_1^\mathrm{red}$, then $t^\mathrm{red} =r$.
		\item Otherwise,  $t^\mathrm{red}= t_1^\mathrm{red}\cdot t_2^\mathrm{red}$.
		\end{enumerate}
		\end{enumerate}
\end{df}

\begin{prop} For every term $t$, the term
	$t^\mathrm{red}$ is reduced and $t\equiv t^\mathrm{red}$.
\end{prop}
\begin{proof}
	By induction on $t$, using Lemma~\ref{May12_3}.
\end{proof}

\begin{prop}\label{fin}
	Assume $\bar{x}=x_1,\ldots,x_n$, $t=t(\bar{x},y)$ is reduced  and  $y$ occurs in $t$. The following are equivalent:
	\begin{enumerate}
		\item If  $M$ is a free Steiner quasigroup,  $\bar{a},b_1,b_2\in M$  and  $t(\bar{a},b_1)= t(\bar{a},b_2)$, then $b_1= b_2$.
		\item If $M$ is a free Steiner quasigroup, $\bar{a}\in M$ is a tuple of independent elements and $b_1,b_2\in M$,  and $t(\bar{a},b_1)= t(\bar{a},b_2)$, then $b_1= b_2$.
		\item If $r_1=r_1(\bar{x})$ and  $r_2= r_2(\bar{x})$  are reduced and  $t(\bar{x},r_1(\bar{x}))^\mathrm{red}= t(\bar{x},r_2(\bar{x}))^\mathrm{red}$, then $r_1\sim r_2$.
	\end{enumerate}
\end{prop}
\begin{proof}
	1 $\Rightarrow$ 2 is clear. 
	
	2 $\Rightarrow$ 1.  Let $\bar{a}= a_1,\ldots,a_n$.  We may assume $M$ has a finite base  $\bar{c}= c_1,\ldots,c_m$.  For each $i$, let  $t_i(\bar{z})$ be a reduced term such that $t_i(\bar{c})= a_i$  and let  $t^\prime(\bar{z},y)= t(t_1(\bar{z}),\ldots,t_n(\bar{z}),y)^\mathrm{red}$. Hence $t^\prime(\bar{z},y)\equiv t(t_1(\bar{z}),\ldots,t_n(\bar{z}),y)$  and $y$ occurs in $t^\prime(\bar{z},y)$. By assumption,  if $b_1,b_2\in M$ and  $t^\prime(\bar{c},b_1)= t^\prime(\bar{c},b_2)$, then $b_1=b_2$. Suppose $t(\bar{a},b_1)=t(\bar{a},b_2)$. Then  $t^\prime(\bar{c},b_1)= t^\prime(\bar{c},b_2)$  and, therefore, $b_1=b_2$.
	
	2 $\Rightarrow$ 3.  Let $t,r_1,r_2$ be as in the assumption and suppose $t(\bar{x},r_1(\bar{x}))^\mathrm{red}= t(\bar{x},r_2(\bar{x}))^\mathrm{red}$. We claim that $r_1\sim r_2$.  Let $M$ be a free Steiner quasigroup  with base $\bar{a}= a_1,\ldots,a_n$, let $b_1= r_1(\bar{a})$ and $b_2 = r_2(\bar{a})$. Let $t_i^\prime(\bar{x})= t(\bar{x},r_i(\bar{x}))^\mathrm{red}$  for $i=1,2$. Then $t_1^\prime\sim t_2^\prime$ and, therefore, $t_1^\prime(\bar{a})= t_2^\prime(\bar{a})$. But $t_i^\prime(\bar{a})= t(\bar{a},b_i)$. Hence  $t(\bar{a},b_1)=t(\bar{a},b_2)$. By the assumption in 2, $b_1=b_2$, that is $r_1(\bar{a})=r_2(\bar{a})$. Then  $r_1\sim r_2$ by Proposition~\ref{May19_1}. 
	
	3 $\Rightarrow$ 2. We may assume  $\bar{a}$ is a base of $M$. Suppose $t(\bar{a},b_1)= t(\bar{a},b_2)$.  Choose reduced terms $r_1,r_2$  such that $b_i=r_i(\bar{a})$  for $i=1,2$. Let $t_i^\prime=t(\bar{x},r_i(\bar{x}))$. Then $t_1^\prime(\bar{a})=t_2^\prime(\bar{a})$.  By Proposition~\ref{May19_1}, $t_1^\prime\sim t_2^\prime$. By assumption 3,  $r_1\sim r_2$, Hence  $b_1=r_1(\bar{a}) =r_2(\bar{a})= b_2$.
\end{proof}

By Lemma~\ref{June9_2}, the equivalent conditions in Proposition~\ref{fin} hold when $t=t(\bar{x},y)$ has only one occurrence of $y$. We conjecture that the conditions hold for every term $t$. This has important consequences for model-theoretic algebraic and definable closures in a free Steiner quasigroup.

\begin{prop}  Assume the equivalent conditions of Proposition~\ref{fin} hold. Let  $\bar{a}$ be a tuple of independent elements in a free Steiner quasigroup $M$  and let $b\in M$. The following are equivalent:
	\begin{enumerate}
		\item $\bar{a}b$ is not independent
		\item $b$ is definable over $\bar{a}$ in $\langle \bar{a}b\rangle$
			\item $b$ is algebraic over $\bar{a}$ in $\langle \bar{a}b\rangle$.
	\end{enumerate}
	Moreover, item 1 implies that $b$ is definable (and hence algebraic) over $\bar{a}$ in $M$.
	\end{prop}
	\begin{proof}
		1 $\Rightarrow$ 2.  Assume $\bar{a}b$ is not independent. By Proposition~\ref{May26_2}, there is a reduced term $t(\bar{x},y)$ containing $y$ such that $t(\bar{a},b)\in\langle \bar{a}\rangle$. Then $t(\bar{a},b) = r(\bar{a})$  for some reduced term $r(\bar{x})$. By the conditions in Proposition~\ref{fin}, $M$  and  $\langle \bar{a},b\rangle$  satisfy the  sentence
		\[\forall y_1y_2 \, \left(t(\bar{a},y_1)= t(\bar{a},y_2)\rightarrow y_1=y_2 \right) \, .\]
It follows that  the equation  $t(\bar{a},y) = r(\bar{a})$ defines $b$ in these structures.
		
		2 $\Rightarrow$ 3  is clear.
		
		3 $\Rightarrow$ 1.  Assume $\bar{a}b$ is independent.  As shown in~\cite{barcas3}, we can elementarily extend $\langle \bar{a},b\rangle $ to a free Steiner quasigroup $N$  with an infinite base  $\{\bar{a}\}\cup\{b_n\mid n\in \omega\}$  with  $b=b_0$.  For every $n$ there is a permutation of the base fixing $\bar{a}$ and sending $b$ to $b_n$, and these permutations can be extended to automorphisms of $N$. Hence, the orbit of $b$ over $\bar{a}$ is infinite in $N$ and $b$ is not algebraic over $\bar{a}$ in $N$. Since the extension is elementary, $b$ is not algebraic over $\bar{a}$ in $\langle \bar{a}b\rangle$.
	\end{proof}

	\noindent{\sc
Sezione di Matematica\\
Universit\`a di Camerino\\
{\tt silvia.barbina@unicam.it}

\noindent{\sc
Departament de Matem\`atiques i Inform\`atica\\
Universitat de Barcelona}\\
{\tt e.casanovas@ub.edu}\\

\end{document}